\documentclass[11pt,leqno]{article}
\usepackage{graphicx, amsfonts, amsthm, amsxtra, amssymb, verbatim}
\usepackage[mathscr]{euscript}
\begin{document}
\newcommand{\nei}[2]{{#1_{(#2)}}}
\newcommand{\qeed}{\hfill$\square$}
\def\O{{\cal O}}
\def\M{{\cal M}}
\def\T{{\cal T}}
\def\N{{\Bbb N}}
\def\U{{\cal U}} 
\def\Z{{\Bbb Z}}
\def\S{{\cal S}}
\def\Q{{\cal Q}}
\def\R{{\Bbb R}}
\def\v{\nu}
\def\C{\Bbb C}
\def\P{\Bbb P}

\def\Qv{{\cal Q}_\v}
\def\nb{ \M/\M^2 }
\def\nbn{ \M'/\M'^2 }
\def\nbv{ \M^{\v}/\M^{\v+1} }

\def\Z{{\Bbb Z}}
\def\p1{{\Bbb P}_1}
\def\pn{{\Bbb P}_n}
\def\F{{\cal F}}
\def\Sf{\F(X,A,L)}
\def\L{{\cal L}}


\newtheorem{theo}{Theorem}
\newtheorem{defi}{Definition}
\newtheorem{cor}{Corollary}
\newtheorem{lem}{Lemma}
\newtheorem{prop}{Proposition}
\newtheorem{rem}{Remark}
\newtheorem{exa}{Example}
\newtheorem{exe}{Exercise}
\newtheorem{conj}{Conjecture}
\title{Foliated neighborhoods of exceptional submanifolds}
\author{C\'esar Camacho and Hossein Movasati}
\date{}

\maketitle
\begin{abstract}

The present article is a study of germs of regular foliations transverse to an embedded strongly exceptional submanifold 
of a complex manifold. Cohomological conditions are given on this 
embedding for the existence of these foliations and their classification is established. One dimensional foliations singular at the submanifold and  a generalization of a linearization theorem of 
Poincar\'e for these foliations, 
are used in this study.  As a consequence of our approach, we obtain a refinement of the  embedding theorem of Grauert.    
    
\end{abstract}

\section{Introduction}
In this paper we classify germs of regular foliations transverse to a negatively embedded compact submanifold of a complex manifold. The dimension of the foliation is assumed to be complementary to that of the embedded submanifold, and natural cohomological hypothesis are imposed on the embedding.
This mathematical object appears naturally, for instance,  in the problem of classification of germs  of complex foliations at a singularity. Indeed, whenever there are infinitely many analytic leaves passing through the singularity we will find in its resolution, irreducible components of the exceptional divisor transverse everywhere to the lifted foliation. On the other hand, this study will lead us to a refinement of the embedding theorem of Grauert \cite{gra62}.

More precisely, consider a complex compact projective  manifold $A$ of dimension $n$ embedded in an $(n+m+1)$-dimensional complex manifold $X$. We denote by $(X,A)$ the germ of the neighborhood of $A$ in $X$. We say that $(X,A)$ is a germ of a {\it foliated neighborhood } of  $A$ if there exists a regular foliation on $X$ of dimension $ m+1$ whose germ at $A$ is transverse to $A$. 
An important tool in this study is the set of one dimensional foliations on $(X,A)$ with singularities at $A$  and normally attracting at $A$. More precisely, a one dimensional foliation on $(X,A)$ is defined by a collection of nontrivial local vector fields $V_i$ defined on open subsets $U_i \subset X,  i\in I$, which are part 
of a covering $(U_i)_{i\in I}$ of $A$, in such a way that for each nonempty intersection $U_i\cap U_j\neq \emptyset$ we have $V_i=f_{ij} V_j$ with $f_{ij}\in\O^*(U_i\cap U_j)$.  
The foliation is $\it{singular}$ at $A$, if $V_i|_{A\cap U_i}= 0$, for each $i\in I$.  Let $\F_1$ be a complex one dimensional foliation on $(X,A)$, singular at $A$. We say that $\F_1$ 
is {\it normally attracting at} $A$ if for each $i\in I$ the linear part of $V_i$ at each $p\in U_i\cap A$, $DV_i(p)$, is a linear operator whose action splits in two invariant 
subspaces $T_p X= T_p A+ N_p$ and $DV_i(p)|N_p$ has eigenvalues  
$\{\lambda_1, ..., \lambda _{m+1}\}\subset \C$,  
whose convex hull does not contain $0\in\C$. Clearly this concept depends only on the foliation and not on the local vector fields. 
The invariant manifold theorem (see \cite{push}) shows that there is, in fact, a foliation $\F_2$ in $(X,A)$, transverse to $A$, whose leaves are of dimension $m+1$, and invariant by $\F_1$. 
We call the pair $(\F_1,\F_2)$ a {\it bifoliation}. Reciprocally we will establish in Theorem \ref{grauertfibrado}, cohomological conditions under which there exists a normally attracting foliation $\F_1$ tangent 
to a given $(m+1)$-dimensional foliation transverse to $A$.
Other natural questions are the existence of such bifoliations and the key question for the classification of these foliations: 
under which conditions $\F_1$ is holomorphically equivalent to its linear part? The linear part is defined by local expressions 
$DV_i(p), p\in A\cap U_i $ and $DV_i(p)=f_{ij}(p)DV_j(p)$ whenever $p\in U_i\cap U_j$, 
on the normal bundle $N$ of rank $m+1$ over $A$. Clearly this equivalence will take the leaves of $\F_2$ to the fibers of $N$.  

The case in which $A$ is a point is classical.  A {\it resonance} among the eigenvalues 
$\{\lambda_1, ..., \lambda _{m+1}\}\subset\C$
is a relation of the kind $\lambda_i=\sum m_j\lambda_j$ where  $m_j\geq 0$ and $\sum m_j\geq 2$.
The theorem of Poincar\'e (see for instance \cite {ilya} or \cite{cer}) states that if $0\in\C^{m+1}$ is an attracting singularity of $\F_1$ and there are no resonances among the eigenvalues of the linear part of $\F_1$ at $0\in\C^{m+1}$
then there is an analytic change of coordinates around $0\in\C^{m+1}$ taking $\F_1$ to its linear
part. This theorem can be extended in the presence of resonances to show the existence of a holomorphic
change of coordinates taking $\F_1$ to a polynomial foliation in normal form involving only the terms
in resonance (see for instance \cite{ilya}). The following gives us a generalization of this theorem to the global situation, i.e. when $A$ is not a point:
\begin{theo}
\label{theo1}
 Let $\F_1$ be a  normally attracting foliation  in a germ of strongly exceptional manifold $(X,A)$. Assume that  
there are no resonances among the eigenvalues of the linear part of $\F_1$ along the normal direction of $A$. 
 If
$$
H^{1}(A, N^{-\nu})=0, \ \nu=1,2,3,\ldots
$$
then there is a biholomorphic map $(X,A)\to (N,A)$, where $N$ is the normal bundle of $A$ in $X$, which is a conjugacy between  $\F_1$ and its linear part in $(N,A)$.
\end{theo}
For a vector bundle $N$ on $A$, and $\mu\in \N$ we write $N^\mu$
to denote the symmetric $\mu$-th power of $N$

For the definition of exceptional and strongly exceptional varieties see Section \ref{10mar03}.
Theorem \ref{theo1} generalizes the linearization theorem proved in \cite{cam} where $A$ is a one dimensional compact curve embedded in a complex surface. Note that in codimension one, that is $m=0$,  we have $\F_1=\F_2$ which is a holomorphic foliation by curves transverse to $A$.

Of particular importance is the case where the germ of $\F_1$ at a point $p\in A$ is a {\it radial singularity} at $p$, that is, all the normal eigenvalues of the linear part of $\F_1$ are equal, which means that after a blow up normal to $A$, the lifted foliation of $\F_1$ becomes a transverse foliation to the blow-up divisor. We call $\F_1$  a {\it radial foliation}. In order to state our next results we need the following cohomological conditions:

(I) Vanishing of cohomologies for arbitrary codimension of $A$ on $X$:

$$
H^{1}(A, N^{-\nu})=0,\  H^1(A, TA\otimes N^{-\nu})=0,\ \nu=1,2,\ldots
$$

(II) If the codimension of $A$ in $X$ is greater than one, then:
$$
H^2(A, \O_A)=0, 
$$
$$
H^1(A, N\otimes N^{-\nu})=0,\   \nu=1,2,\ldots.
$$
The following theorem gives cohomological conditions for the existence of radial foliations:
\begin{theo}
\label{theo2}
Let $(X,A)$ be a germ of strongly exceptional manifold satisfying the cohomological conditions (I) and (II).
Then there exists a germ of radial foliation in $(X,A)$
 \end{theo}
Combining Theorem \ref{theo1} and Theorem \ref{theo2} we obtain the following generalization to any codimension of the embedding theorem of Grauert in \cite{gra62}. 
\begin{cor}
 \label{theo3}
 Let $(X,A)$ be a germ of strongly exceptional manifold satisfying the cohomological conditions (I) and (II). Then, 
the germ of embedding of $A$ in $X$ is biholomorphic to the germ of embedding of $A$ in $N$.
\end{cor}

Let us restrict to the case in which $A$ is a Riemann surface and $N$ is direct sum of $m+1$ line bundles $N=L_1\oplus L_2\oplus\cdots \oplus L_{m+1}$. In this case the Serre duality implies that the cohomological condition (I) is equivalent to say that $\Omega^1\otimes N^\nu$ and $\Omega^1\otimes \Omega^1 \otimes N^{\nu}$ have no global sections, where $\Omega^1$ is the cotangent 
bundle of $A$. 
We have
$$
N^\nu=\oplus_{i_1+i_2+\cdots+i_{m+1}=\nu, \ i_j\geq 0}\ \ L_1^{i_1}\otimes L_2^{i_2}\otimes\cdots\otimes L_{m+1}^{i_{m+1}}
$$
and so (I) together with the strongly exceptional property  follows from
$$
c(L_i)<0,\ c(L_i) <4-4g, \  i=1,2,\ldots,m+1. 
$$
In a similar way the condition (II) is equivalent to say that $A\cong \P ^1$ and:
$$
|c(L_i)-c(L_j)|\leq 1,\ i,j=1,2,\ldots,m+1. 
$$
In this case the decomposition of the normal bundle is automatic and it is called Birkhoff theorem. 
From this we obtain as a corollary the following result of Laufer  \cite{lau79}:
\begin{cor}
 If $\P ^1\subset  X$ is strongly exceptional and $c(L_i)<0,\ |c(L_i)-c(L_j)|\leq 1,\ i,j=1,2,\ldots m+1$, where $L_i$'s are line bundles which appear in the decomposition of 
the normal bundle of $A$ in $X$, then the germ $(X,\P ^1)$ is biholomorphic to the germ $(N,\P ^1)$. 
\end{cor}
In the case in which the codimension of $A$ in $X$ is greater than one the condition (II) seems to be necessary for our theorem. It imposes conditions on the submanifold $A$ itself apart from 
negativity conditions on the normal bundle $N$. It would be of interest to show that, for instance, the Grauert theorem does not hold for Riemann surfaces of genus greater than zero and 
codimension greater than one.

 The embedding theorem of Grauert \cite{gra62} states that under the cohomological condition (I) on a codimension one embedding
there is a neighborhood of $A\subset X$ which is biholomorphically equivalent to a neighborhood of the zero section in the normal bundle $N$ to $A$ in $X$. 
The methods used in this paper give the following generalization of this theorem:

\begin{theo}
\label{grauertfibrado}
 Let $\F_2$ be a transverse regular foliation  of dimension $m+1$ in a germ of strongly exceptional manifold $(X,A)$.  Assume that (I) and (II) hold.  Then there is a biholomorphic map $(X,A)\to (N,A)$, where $N$ is the normal bundle of $A$ in $X$, which conjugates $\F_2$ with the foliation in $(N,A)$ given by the fibers of $N$.
\end{theo}

The paper is organized as follows: In Section \ref{10mar03} we review some facts about exceptional varieties. 
In Section \ref{restriction} we prove the key Proposition of the present text. It establishes cohomological conditions under which the restriction of line bundles from $X$ to $A$ is injective.
The blow up process along $A$, reduces our problems in an arbitrary codimension to the codimension one case. This is explained in Section \ref{blowup}. 
Section \ref{pr2} is dedicated to the proof of Theorem \ref{theo1}.
In Sections \ref{linefield} and \ref{linefield2} we prove Theorem \ref{theo2}. Finally, in section \ref{pr3} we prove Theorem \ref{grauertfibrado}.

\section{Grauert's vanishing theorem} 
\label{10mar03}
We start this section with some basic definitions.
Let $X$ be an analytic variety and $A$ be a
compact connected subvariety of $X$. We say that $A$  is
 {\it exceptional} 
in $X$ if there exists an
 analytic variety $X'$ and a proper
 surjective holomorphic map $\Phi:X\rightarrow X'$
  such that
\begin{enumerate}
\item
$\phi(A)=\{p\}$ is a single point;
\item
$\phi:X-A\rightarrow X'-\{p\}$ is an
analytic isomorphism;
\item
For small neighborhoods $U'$ and $U$ of
$p$ and $A$, respectively, $\O_{X'}(U')\rightarrow
\O_X(U)$ is an isomorphism.
\end{enumerate} 
We also say that $A$ can be blown down to a point or 
is {\it contractible} or {\it negatively embedded}.
The vector bundle $V\rightarrow A$ over a complex manifold $A$
 is called negative (in the sense of Grauert) if its zero section
 is an exceptional variety in $V$. Naturally $V\rightarrow A$ is called
 positive if its dual is negative. Let $X$ be a smooth variety and let $A$ be a smooth subvariety. We say that the germ $(X,A)$ is {\it strongly 
exceptional} if it is exceptional and the normal bundle of $A$ in $X$ is negative.

 Let $A$ be a complex compact manifold and $N$ be a negative line bundle on $A$. 
This is equivalent to say that $N^{-1}$ is a positive
line bundle in the sense of Kodaira. Kodaira vanishing theorem says that for any coherent sheaf $\S$ on $A$ there is $\v_0\in \N$ such that
\begin{equation}
\label{kodaira}
H^\mu(A,\S\otimes N^{-\v})=0,\ \v\geq \v_0,\ \mu=1,2,\ldots.  
\end{equation}
Let us now be given a subvariety $A$ of a variety $X$.
Let $\M$ be the sheaf of holomorphic functions in $(X,A)$ which vanish at $A$ and let 
$\S$ be a coherent sheaf in $(X,A)$. For $\v\in\N$, the sheaf $\S(\v):=\S\M^\v/\S\M^{\v+1}$ is a coherent sheaf with support $A$ and
in fact:
$$
\S(\v)\cong \tilde \S\otimes N^{-\v}.
$$
where $\tilde \S=\S(0)$ is the structural restriction of $\S$ to $A$.  
If there is no danger of confusion we will also use $\S$ to denote $\tilde \S$, being clear from the text which we mean. 
\begin{theo}
\label{25mar01}
(Grauert \cite{gra62},Satz 2, p. 357) 
Let us be given a strongly exceptional submanifold $A$ of 
a manifold $X$. There exists a
positive integer $\v_0$ such that
\[
H^{\mu}(U, \S\M^\v)=0, \ \ \mu\geq 1, \ \v\geq \v_0
\]
where $U$ is a small strongly pseudoconvex neighborhood of 
$A$ in $X$. Moreover, $\v_0$ in the above theorem can be taken smaller than  the same number $\v_0$ in (\ref{kodaira}).
\end{theo}

\section{Restriction of line bundles}
\label{restriction}
\begin{prop}
\label{nashenas}
Let $A$ be a strongly exceptional complex manifold of dimension $n$ embedded   in a  manifold $X$ of dimension $n+1$.
Moreover, suppose that 
$$
H^1(A,N^{-\nu})=0,\ \ \nu=1,2,3,\ldots
$$ 
where $N$ is the normal 
bundle of the embedding and $N^{-1}$ is the dual bundle. 
The restriction map 
$$
r:H^1(X,\O_X^*)\rightarrow H^1(A,\O_A^*)
$$ 
is injective. 
\end{prop}
\begin{proof}
The sheaf of holomorphic
sections of $N^{-\nu}$ is isomorphic to $\M^{\nu}/\M^{\nu+1}$ and so we have
$$
H^1(A, \M^{\nu}/\M^{\nu+1})=0, \ \forall \nu\in \N.
$$
The submanifold $A$ is strongly exceptional in $X$  and so by Theorem \ref{25mar01} applied to $\S=\O_X$ we have
\[
H^1(U, \M)=0,
\]
where $U$ is a strongly pseudoconvex neighborhood of $A$ in $X$.
The diagram
\begin{equation}
\begin{array}{ccccccccc}
 & & & & 0  & & & &  \\
& & & & \downarrow & & & &  \\
& & & & \M & & & &  \\
& & & & \downarrow & & & &  \\
0 &\rightarrow & \Z & \rightarrow&\O_X &\rightarrow &\O_X^* &\rightarrow & 0  \\
& &\downarrow & &\downarrow & &\downarrow & &  \\
0 &\rightarrow & \Z & \rightarrow&\O_A &\rightarrow &\O_A^* &\rightarrow & 0  \\
& & & & \downarrow & & & &  \\
 & & & & 0  & & & &  \\
\end{array}
\end{equation}
gives us 
\begin{equation}
\begin{array}{ccccccccc}

& &  H^1(U, \M)=0& & & & & &  \\
& & \downarrow& &  & & & &  \\
H^1(U,\Z) &\rightarrow & H^1(U, \O_X) & \rightarrow& H^1(U, \O_X^*)
&\rightarrow & H^2(U,\Z)   \\
\downarrow & &\downarrow & &\downarrow & &\downarrow & &  \\
 H^1(A, \Z) &\rightarrow & H^1(A, \O_A) & \rightarrow& H^1(A, \O_A^*)
 &\rightarrow &H^2(A,\Z)  \\
\end{array}
\end{equation}
By considering a smaller neighborhood $U$, if necessary, we can assume
that $A$ and $U$ have the same topology and so the first and fourth
column functions are isomorphisms. In the argument which we are going
to consider now we do not mention the name of mappings,
 being clear from the above diagram which mapping we mean. 

Let us consider $x_1\in H^1(U, \O_X^*)$ which is mapped to the trivial bundle in
$H^1(A, \O_A^*)$. Since the fourth column is an isomorphism, $x_1$ maps
to zero in $H^2(U,\Z)$. This means that there is a $x_2\in 
H^1(U, \O_X)$ which maps to $x_1$. Let $x_3$ be the image of $x_2$ in
$H^1(A, \O_A)$. Since the above diagram is commutative, $x_3$ maps
to the trivial bundle in $H^1(A,\O^*_A)$. Therefore there exists
a $x_4$ in $H^1(A,\Z)$ which maps to $x_3$. Since the first column
is an isomorphism and the second is injective, we conclude that
$x_4\in H^1(U,\Z)\cong H^1(A,\Z)$ maps to $x_2$ and so $x_2$ maps to
$x_1=0$ in $H^1(U, \O_X^*)$.  
\end{proof}

Now, we give some applications of Proposition \ref{nashenas}. Let us assume that $(X,A)$ has a transverse foliation namely $\F$. 
The normal bundle $N$
of $A$ in $X$
 has a meromorphic global section namely $s$. Let
\[
{\rm div}(s)=\sum n_iD_i,\ n_i\in\Z.
\]
We define the divisor $D$ in $X$ as follows:  
\begin{equation}
 \label{omidimpa}
D=A-\sum n_i\tilde D_i,
\end{equation}
where $\tilde D_i$  is the saturation of $D_i$ by $\F$.  
The line bundle $L_D$ associated to $D$  restricted to $A$
is the trivial line bundle, and so by Proposition \ref{nashenas},
$L_D$ is trivial or equivalently
\begin{prop}
\label{ilan}
Under the hypothesis of Proposition \ref{restriction}, there exists a meromorphic function $g$ on $(X,A)$ with
$$
{\rm div}(g)=D
$$
where $D$ is given by (\ref{omidimpa}).
\end{prop}

We give another application of Proposition \ref{nashenas}. 


\begin{theo}
\label{27jan01}
Let $A$ be a strongly exceptional codimension one submanifold  of $X$. Further, assume that
\begin{equation}
\label{omidhaa}
H^1(A, N^{-\nu})=0,\ \forall \nu=1,2,\ldots
\end{equation}
Any transverse holomorphic foliation in $(X,A)$ is biholomorphic to the canonical transverse foliation of $(N,A)$ by the fibers of $N$. In particular,  the germs of any two holomorphic  transverse foliations  in $(X,A)$ are equivalent.    
\end{theo}
 This theorem in the case in which $A$ is a Riemann surface  is proved in 
\cite{cam}. 
\begin{proof}
Let $\F$ be the germ of a transverse foliation in $(X,A)$ and
$N$ the normal bundle of $A$ in $X$. Let also $\F'$ be the canonical transverse foliation of $(N,A)$. 
Let $g$ (resp. $g'$) be the meromorphic function constructed in Proposition \ref{ilan} for the pair $(X,A)$ resp. $(N,A)$.
 We claim that at each point $a\in A$ there exists a unique biholomorphism
\[
\psi_a:(X,A,a)\rightarrow (N,A,a)
\]
with the following properties: 
\begin{enumerate}
\item
$\psi$ induces the  identity map on $A$;
\item
$\psi$ sends $\F$ to $\F'$; 
\item
The pullback of $g'$ by $\psi$ is $g$. 
\end{enumerate}
The uniqueness property implies that these local biholomorphisms are restrictions of a global biholomorphism 
$\psi:(X,A)\rightarrow (N,A)$ which sends $\F$ to $\F'$.

Now we prove our claim. Fix a coordinates system $x=(x_1,x_2,\cdots,x_n)$ in a
neighborhood of $a$ in $A$. We extend $x$ to a coordinates system $(x,x_{n+1})$
of a neighborhood of $a$ in $X$ such that $A$ (resp. $\F$) in this 
coordinates system is given by $x_{n+1}=0$ (resp. $dx_i=0,\ i=1,2,\ldots,n$). In this coordinate system 
$$
g(x,x_{n+1})=Q(x)x_{n+1}f(x,x_{n+1}), 
$$ 
where $Q(x)$ is a meromorphic function in a neighborhood of $a$ in $A$ and it does not depend on the choice of an embedding of $A$ and $f$ is a holomorphic function in $(X,a)$ without zeros.
By changing the coordinates in $x_{n+1}$ we can assume that $f=1$.
 It is easy to check that the coordinate system $(x,x_{n+1})$ is unique and it gives us the local biholomorphism $\psi_a$. 
\end{proof}

\section{Blow up along a submanifold}
\label{blowup}
Let $N$ be a vector bundle of rank $m+1$ over $A$ and let $\tilde A:=\P(N)$ be the projectivization of the fibers of $N$. We have a canonical projection map 
$\pi : \tilde A\to A$ with fibers isomorphic to $\P ^m$. The space $\tilde A$ carries a distinguished line bundle $\tilde N$ which is defined by:
$$
\tilde N_x=\text{ the line representing $x$ in the vector space $N_{\pi(x)}$},\ \ x\in\tilde A
$$
In some books the notation $\O_{\tilde A}(-1)$  is used to denote the sheaf of sections of $\tilde N$ because the line bundle $\tilde N$ is the tautological bundle  
restricted to the fibers of $\pi$. It has the following properties:
$$
\pi_*(\O(\tilde N^{-\nu}))\cong \O(N^{-\nu}),\ \nu=0,1,2,\ldots
$$
$$
\pi_*(\O(\tilde N^{\nu}))=0,\ \ \nu=1,2,\ldots
$$
$$
H^q({\tilde A}, \pi^*(\S)\otimes \O(\tilde N^{-\nu}))\cong H^q(A,  \S\otimes \O(N^{-\nu})),
\ \nu=1,2,\ldots
$$
for every locally free sheaf $\S$ on $A$ (see \cite{grpere}, p. 178). Here $\O$ of a bundle means the sheaf of its sections. 
When there is no ambiguity between a bundle and the sheaf of its sections we do not write $\O$. We will also use the following: if for a sheaf of abelian groups $\S$
on $\tilde A$ we have $R^i\pi_*(\S)=0$ for all $i=1,2,\ldots$, then
$$
H^i(\tilde A,\S)\cong H^i(A,\pi_*\S),\ i=0,1,2,\ldots.
$$  
We will apply this for the sheaf of sections of $T\P^m\otimes \tilde N^{-\v},\ \v=1,2,\ldots$, where $T\P^{m}$ is the subbundle of $T\tilde A$ corresponding to vectors tangent to the fibers of $\pi$. 

By definition $\tilde N$ is a subbundle of $\pi^*N$ and we have the short exact sequence:
\begin{equation}
0\to \tilde N\to \pi^*N\to T\P^m\to 0.
\end{equation}
 We take $\O$ of the above sequence, make a tensor product with $\O(\tilde N^ {-\v}),\ \v=1,2,\ldots$ and 
apply $\pi_*$: we get
\begin{equation}
 \label{18nov}
0\to N^{-\v+1}\to N\otimes N^ {-\v}\to \pi_*(T\P^m\otimes \tilde N^ {-\v})\to 0 
\end{equation}
(for simplicity we have not written $\O(\cdots)$).
Note that $R^1\pi_*\O(\tilde N^{-\v+1})=0,\ \v=1,2,\ldots$. Note also that if $N$ is not a line bundle then  $N\otimes N^ {-1}$ may not be the trivial bundle.

The vector bundle $T\P^m$ appears also in the short exact sequence:
\begin{equation}
\label{18nov2010}
0\to \O(T\P^m)\to \O(T\tilde A) \to \pi^{*}\O(TA)\to 0,
\end{equation}
where $\O(T\tilde A) \to \pi^{*}\O(TA)$ is the map obtained by derivation of $\tilde A\to A$ and then considering the pull-back of $\O(TA)$.

Let $A$ be a compact submanifold of $X$ with
$$
n=\dim(A),\ m+1=\dim(X)-n.
$$
and let $N=TX\mid_A/TA$ be the normal bundle of $A$ in $X$. We make the blow up of $X$ along $A$:
$$
\pi :\tilde X\to X,\ \tilde A:=\pi^{-1}(A)= \P(N). 
$$
The normal bundle of $\tilde A$ in $\tilde X$ is in fact:
$$
\tilde N=N_{\tilde X/\tilde A}\cong \O_{\tilde A}(-1).
$$
Combining all these with Proposition \ref{nashenas}, we get the same proposition without the codimension restriction: 
\begin{prop}
\label{nashenas1}
Let $A$ be a strongly exceptional complex submanifold of $X$.
Moreover, suppose that 
$$
H^1(A,N^{-\nu})=0,\ \ \nu=1,2,3,\ldots
$$ 
where $N$ is the normal 
bundle of the embedding and $N^{-1}$ is the dual bundle. 
The restriction map 
$$
r:H^1(X,\O_X^*)\rightarrow H^1(A,\O_A^*)
$$ 
is injective. 
\end{prop}


\section{Proof of Theorem \ref{theo1}}
\label{pr2}
First we prove that there is a holomorphic vector field $V$ on $(X,A)$ tangent to the foliation $\F_1$ and singular at $A$. 
Indeed, by our hypothesis such a vector field exists locally. Thus there is a finite covering $(U_i)_{i\in I}$ of $(X,A)$ and for each $i\in I$ a vector field $V_i$ on $U_i$ such that at any 
$p\in A\cap U_i$,  $DV_i(p)$ has $n$ eigenvalues equal to zero ( along the direction of $A$), and eigenvalues  $\{\lambda_1, ..., \lambda _{m+1}\}$ whose convex hull does not contain 
$0\in \bold C$. On each nonempty intersection $U_i\cap U_j\neq\emptyset$ we have $V_i=f_{ij}V_j$, where the cocycle $L=\{f_{ij}\}$ is a line bundle. We write  the linear part of 
$V_i=f_{ij}V_j$ and we conclude that  
$f_{ij}|_A=1$. This means that the restriction of $L$ to $A$ is the trivial bundle. The collection of vector fields $V_i$, $i\in I$, defines a global section of $TX\otimes L$ and by Proposition 3,  $L$ is a trivial bundle. 

On the other hand, if $V$ and $\tilde V$ are vector fields tangent to $\F_1$ on $(X,A)$ and to its linear part $\tilde \F_1$ on $(N,A)$, respectively, by Poincar\'e theorem we know that locally there exists a unique biholomorphism $f_p:(X,A,p)\to (N,A,p)$ conjugating $V$ to $\tilde V$. Since the $f_p$'s are unique we conclude that they coincide in their common domains of definition, and hence, they give us a biholomorphism $f:(X,A)\to (N,A)$ conjugating $V$ to $\tilde V$.

\section{Proof of Theorem \ref{theo2}, codimension one}
\label{linefield}
In this section $A$ is a codimension one submanifold of $X$, $N$ is the normal bundle of $A$ in $X$ and $TA$ is the tangent bundle of $A$. 
\begin{prop}
\label{25sept09}
Assume 
\begin{equation}
 \label{naturalizacao}
H^1(A,N^{-1}\otimes TA)=0.
\end{equation} 
Then the pair $(TA\subset TX|_{A})$ is split, that is 
$$
TX|_A\cong N\oplus TA.
$$ 
\end{prop}
\begin{proof}
It is enough to construct a vector bundle morphism $Y:N\to TX|_A$ with the image transverse to $TA$.  First, we construct $Y$ locally, i.e. we find $Y_i: N|_{U_i} \rightarrow TX|_{U_i}$ with the desired property for an open covering $U_i, \ i\in I$ of $A$. Let $\tilde Y_i$ be the composition $N|_{U_i} \rightarrow TX|_{U_i}\to N|_{U_i}$. Then $\tilde Y_i=a_{ij}\tilde Y_j$, where  $\{a_{ij}\}\in H^1(A,\O^*_A)$ is a line bundle. Now, $\tilde Y_i$'s are sections of the trivial bundle $N^{-1}\otimes N$ with no zeros and so $\{a_{ij}\}$ is a trivial bundle and so  we can assume that $\tilde Y_i=\tilde Y_j$. Now
$$
\{Y_{ij}\}:=\{Y_i-Y_j\}\in H^1(A, {\rm Hom}(N,TA)).
$$
Since ${\rm Hom}(N,TA)\cong N^{-1}\otimes TA$, our assertion follows by the vanishing hypothesis (\ref{naturalizacao}). 
\end{proof}
If $A$ is a curve  then we can use the Serre duality and the cohomological condition (\ref{naturalizacao}) follows from:
$$
A\cdot A<4-4g.
$$
Let $\F$ be a non singular transverse foliation by curves in $(X,A)$. We have the canonical embedding 
$$
T\F|_A\cong N\hookrightarrow TX|_A,
$$
In Proposition \ref{25sept09} we constructed a transverse embedding $N\to TX|_A$  and it is natural to ask whether it comes from a holomorphic foliation as above.
\begin{prop}
\label{31aug09}
 Assume  that $A$ is strongly exceptional codimension one submanifold of  $X$ and 
\begin{equation}
\label{1oct09}
H^1(A, N^{-\nu}\otimes TX\mid_A)=0, \ \ \nu=2,3,\ldots
\end{equation}
Any transverse embedding  $N\to TX|_A$ is associated to  a non singular  transverse foliation $\mathcal F$ defined in a neighborhood of $A$.
\end{prop}
\begin{proof}
We take local sections of $N$ which trivialize $N$ and have no zero point. The images of these sections under 
$N\subset TX|_A$ can be extended to vector fields  $X_i$ defined in $U_i,i\in I$, where $\{U_i\}_{i\in I}$  is a covering of $(X,A)$.
Therefore, 
$$
X_i|_A=f_{ij}X_j|_A,\ N^{-1}=\{f_{ij}\}.
$$
The normal bundle $N$ of $A$ in $X$ extends to a line bundle $\tilde N$ in $(X,A)$ as follows: We take local holomorphic functions $f_i$ in $(X,A)$ such that  $A=\{f_i=0\}$. Now $f_i=\tilde f_{ij}f_j$ and $\tilde N=\{\tilde f_{ij}\}$ is a 
line bundle in  $(X,A)$ which restricted to $A$ is the normal bundle. Now
$$
\{\Theta_{ij}\}=\{X_i-\tilde f_{ij}X_j\}\in H^1(X, \M_A \otimes TX \otimes N^{-1} ).
$$ 
By our hypothesis and Theorem \ref{25mar01} 
the cohomology group in the right hand side is zero.  
\end{proof}

 Using the long exact sequence of $$
0\to TA\otimes N^{-\nu}\to TX|_A\otimes N^{-\nu}\to N^{-\nu+1}\to 0,
$$
 one can see easily that the hypothesis (\ref{1oct09}) together with (\ref{naturalizacao}) follows from:
\begin{equation}
\label{01oct09}
H^1(A,N^{-v}\otimes TA)=0,\ \ H^1(A,N^{-v})=0,\ \nu=1,2,\ldots
\end{equation}
For the case in which $A$ is a Riemann surface, we use Serre duality and (\ref{01oct09}) follows from:
$$
A\cdot A<4-4g \hbox{ for }g\geq 1 \hbox{   and   } A\cdot A< 2 \hbox{ for } g=0.
$$
In this case, Propositions \ref{25sept09} and \ref{31aug09} and their generalization to foliations with tangencies were proved in \cite{ho19}.


\section{Proof of Theorem \ref{theo2}, codimension greater than one}
\label{linefield2}
We perform blow-up along $A$. Recall the notation introduced in Section \ref{blowup}. We would like to construct a transverse holomorphic foliation in 
$(\tilde X,\tilde A)$. This is already done in the previous section. We need the cohomological conditions:
\begin{equation}
H^1(\tilde A,\tilde N^{-\v}\otimes T\tilde A)=0,\ \ H^1(\tilde A,\tilde N^{-\v})=0,\ \nu=1,2,\ldots
\end{equation}
Now, we would like to translate all these in terms of the data of the embedding $A\subset X$. First, note that
$$
H^1(\tilde A,\tilde N^{-\v})\cong H^1(A,N^{-\v}).
$$
We make the tensor product of the sequence (\ref{18nov2010}) with $\tilde N^\v$ and write the long exact cohomology sequence. We conclude that if
$$
H^1(\tilde A, T\P ^m\otimes \tilde N^{-\v})=0,\ H^1(A, TA\otimes N^{-\v})=0,\ \ \nu=1,2,\ldots.
$$
then 
$$
H^1(\tilde A, T\tilde A\otimes \tilde N^{-\v})=0,\ \ \nu=1,2,\ldots
$$
Since $R^1\pi_* (T\P ^m\otimes \tilde N^{-\v})=0, \ \v=1,2,\ldots$, we have 
$$
H^1(\tilde A, T\P ^m\otimes \tilde N^{-\v})=H^1(A, \pi_*(T\P ^m\otimes \tilde N^{-\v})).
$$ 
We write the long exact sequence of (\ref{18nov}) and conclude that if 
$$
H^1(A, N\otimes N^{-\v})=0,\ H^2(A, N^{-\v+1})=0 ,\ \  \v=1,2,\ldots 
$$
then
$$
H^1(\tilde A, T\P ^m\otimes \tilde N^{-\v})=0,\ \ \v=1,2,\ldots 
$$
Finally we conclude that if
$$
H^1(A, N\otimes N^{-\v})=0,\ H^2(A, N^{-\v+1})=0  ,\ \ H^1(A, TA\otimes N^{-\v})=0 ,\  \nu=1,2,\ldots.
$$
then
$$
H^1(\tilde A, T\tilde A\otimes \tilde N^{-\v})=0,\ \nu=1,2,\ldots.
$$
\section{Proof of Theorem \ref{grauertfibrado}}
\label{pr3}
Using Theorem \ref{theo1}, it is enough to construct a second foliation $\F_1$ such that $(\F_1,\F_2)$ is a germ of radial bifoliation.
In codimension one, we have $\F_1=\F_2$ and so we can assume that $m>0$. After performing a blow-up along $A$  our problem is reduced to the following one: 
Let $\tilde A$ be a codimension one submanifold of $\tilde X$  and let $\tilde \F_2$ be a $(m+1)$-dimensional regular foliation in $X$ transverse to $A$. The transversality implies that $\tilde \F_2\cap \tilde A$ is a regular foliation of dimension $m$ in $\tilde A$.  In fact it is the foliation by the blow up divisors $\P ^m$. Its tangent bundle is denoted by $T\P ^m$ in Section \ref{blowup}. 
We would like to construct a transverse to $\tilde A$ foliation $\tilde \F_1$ of dimension one such that its leaves are contained in the leaves of 
$\tilde \F_2$. The proof is a slight modification of Proposition \ref{25sept09} and Proposition \ref{31aug09}. 
In both proposition $T X|_A$ is replaced with $T \tilde \F_2|_{\tilde A}$ and $T A$ is replaced 
with $T\P ^m$. In Proposition \ref{25sept09}, the cohomological condition is 
$$
H^1(\tilde A,\tilde N^{-1}\otimes T\P ^m)=0.
$$
which follows from the condition (II).


\begin{thebibliography}{1}

\bibitem{cam}
C\'esar Camacho, Hossein Movasati, and Paulo Sad.
\newblock Fibered neighborhoods of curves in surfaces.
\newblock {\em J. Geom. Anal.}, 13(1):55--66, 2003.


\bibitem{cer}
Dominique Cerveau.
\newblock Singular points of differential equations: on a theorem of Poincar\'e. The scientific legacy of Poincar\'e.
\newblock 99--112, Hist. Math., 36, Amer. Math. Soc., Providence, RI, 2010. 




\bibitem{grpere}
H.~Grauert, Th. Peternell, and R.~Remmert, editors.
\newblock {\em Several complex variables. {VII}}, volume~74 of {\em
  Encyclopaedia of Mathematical Sciences}.
\newblock Springer-Verlag, Berlin, 1994.
\newblock Sheaf-theoretical methods in complex analysis, A reprint of {{\i}t
  Current problems in mathematics. Fundamental directions. Vol. 74} (Russian),
  Vseross. Inst. Nauchn. i Tekhn. Inform. (VINITI), Moscow.

\bibitem{gra62}
Hans Grauert.
\newblock \"{U}ber {M}odifikationen und exzeptionelle analytische {M}engen.
\newblock {\em Math. Ann.}, 146:331--368, 1962.

\bibitem{push}
M.~W. Hirsch, C.~C. Pugh, and M.~Shub.
\newblock {\em Invariant manifolds}.
\newblock Lecture Notes in Mathematics, Vol. 583. Springer-Verlag, Berlin,
  1977.

\bibitem{lau79}
Henry~B. Laufer.
\newblock On {${\bf C}P^{1}$} as an exceptional set.
\newblock In {\em Recent developments in several complex variables ({P}roc.
  {C}onf., {P}rinceton {U}niv., {P}rinceton, {N}. {J}., 1979)}, volume 100 of
  {\em Ann. of Math. Stud.}, pages 261--275. Princeton Univ. Press, Princeton,
  N.J., 1981.

\bibitem{ilya}
Yulij Ilyashenko, Sergei Yakovenko.
\newblock Lectures on analytic differential equations.
\newblock Graduate Studies in Mathematics, 86. American Mathematical Society, Providence, RI, 2008.


\bibitem{ho19}
Hossein Movasati and Paulo Sad.
\newblock Embedded curves and foliations.
\newblock Publ. Mat. 55, 401--411, 2011. 

\end{thebibliography}

\def\cprime{$'$} \def\cprime{$'$} \def\cprime{$'$}

\end{document}